\documentclass[11pt,a4paper]{article}
\usepackage{amsmath}
\usepackage{a4wide}
\usepackage{mathrsfs}
\usepackage{stmaryrd}
\usepackage{pifont}
\usepackage{amssymb}
\usepackage{color}
\usepackage{amsthm}
\usepackage{graphicx}
\usepackage[colorlinks=true]{hyperref}
\hypersetup{linkcolor=blue, urlcolor=red, citecolor=red}
\allowdisplaybreaks[4] \numberwithin{equation}{section}
\newtheorem{theorem}{Theorem}[section]

\newtheorem{lemma}{Lemma}[section]
\newtheorem{proposition}{Proposition}[section]
\newtheorem{remark}{Remark}[section]

\newtheorem{corollary}[theorem]{Corollary}

\input amssym.def
\input amssym
\title{Equivalence of minimal time and minimal norm control problems for semilinear heat equations}
\author{Huaiqiang Yu\thanks{School of Mathematics and Statistics, Wuhan University,
    Wuhan, 430072, P. R. China. Email: huaiqiangyu@whu.edu.cn}}
\date{}
\begin{document}
\maketitle
\begin{abstract}
    In this paper, we establish the equivalence of minimal time and minimal norm control problems for semilinear heat equations in which the controls are distributed internally in an open subset of the state domain. As an application, the Bang-Bang property for minimal norm controls are also presented.
\end{abstract}
\section{Introduction}
    Let $\Omega\subset\mathbb{R}^N$ ($N\in\mathbb{N}$) be a bounded domain with a smooth boundary $\partial\Omega$ and $\omega$ be an open and nonempty subset of $\Omega$. Denote by $\chi_\omega$ the characteristic function of the set $\omega$. Let $T$ be a positive number and write $\mathbb{R}^+\equiv(0,+\infty)$.
\par
    In the present paper, we consider the following two controlled heat equations
\begin{equation}\label{e-1-1}
\begin{cases}
    y_t-\triangle y+f(y)=\chi_\omega u&\mbox{in}\;\;\Omega\times\mathbb{R}^+,\\
    y=0&\mbox{on}\;\;\partial\Omega\times\mathbb{R}^+,\\
    y(0)=y_0&\mbox{in}\;\;\Omega
\end{cases}
\end{equation}
    on the time interval $[0,+\infty)$ and
\begin{equation}\label{e-1-2}
\begin{cases}
    y_t-\triangle y+f(y)=\chi_\omega v&\mbox{in}\;\;\Omega\times(0,T),\\
    y=0&\mbox{on}\;\;\partial\Omega\times(0,T),\\
    y(0)=y_0&\mbox{in}\;\;\Omega
\end{cases}
\end{equation}
    on the finite time interval $[0,T]$, where the initial state $y_0$ is assumed to be a nontrivial function in $L^2(\Omega)$, and $u$ and $v$ are the controls taken accordingly from the spaces $L^\infty(\mathbb{R}^+;L^2(\Omega))$ and $L^\infty(0,T;L^2(\Omega))$. The solutions of (\ref{e-1-1}) and (\ref{e-1-2}), denoted by $y(\cdot;u,y_0)$ and $y(\cdot;v,y_0)$, are considered to be  functions of the time variable $t$ from $[0,+\infty)$ and $[0,T]$ to the space $L^2(\Omega)$, respectively.
\par
    Let $r>0$ be a constant. For each $T>0$ and each $M>0$, we define the following two admissible sets of controls:
\vskip 5pt
    \noindent$\mathcal{U}_M=\{u\in L^\infty(\mathbb{R}^+;L^2(\Omega)):\|u(\cdot)\|_{L^2(\Omega)}\leq M\;\;\mbox{a.e. in}\;\;\mathbb{R}^+\\
    ~~~~~~~~~~~~~~~~~~~~~~~~~~~~~~\mbox{and}\;\;\exists t>0\;\;\mbox{s.t.}\;\;y(t;u,y_0)\in B(0,r)\}$;
\vskip 5pt
    \noindent$\mathcal{V}_T=\{v\in L^\infty(0,T;L^2(\Omega)): y(T;v,y_0)\in B(0,r)\}$;
\vskip 5pt
    \noindent where $B(0,r)$ is the closed ball in $L^2(\Omega)$ centered at the original point and of radius $r$.
\par
    In this paper, we assume that
\vskip 5pt
    \noindent $(H_1)$ $f:\mathbb{R}\to\mathbb{R}$ is continuously differentiable with $|f'(y)|\leq L$ and $f(y)y\geq 0$ for any $y\in\mathbb{R}$, where $f'(y)$ is the derivative of $f$ in $y\in\mathbb{R}$ and $L>0$ is a constant;
\vskip 5pt
    \noindent $(H_2)$ The initial state $y_0$ satisfies $y_0\notin B(0,r)$.
\vskip 5pt
    \noindent
    It is obvious that the assumption $(H_1)$ implies that $f(0)=0$. Under the assumptions
    $(H_1)$ and $(H_2)$, it is well known that for each $u\in\mathcal{U}_M$ and each $y_0\in L^2(\Omega)$,  Equation
    (\ref{e-1-1}) has a unique solution $y(t;u,y_0)$ in $C([0,T];L^2(\Omega))\cap L^2(0,T;H_0^1(\Omega))$ (see Page 500, Chapter 9 in \cite{b14}). Moreover, for each $t\in\mathbb{R}^+$
\begin{equation}\label{p-1-2}
    \|y(t;0,y_0)\|_{L^2(\Omega)}\leq e^{-\lambda_1t}\|y_0\|_{L^2(\Omega)},
\end{equation}
    where $\lambda_1>0$ is the first eigenvalue for the operator $-\triangle$ with the domain $D(-\triangle)=H_0^1(\Omega)\cap H^{2}(\Omega)$.
    The proof of (\ref{p-1-2}) will be given in the appendix of our paper. From this decay property of Equation (\ref{e-1-1}) with $u\equiv0$, we know that the set $\mathcal{U}_M$ is nonempty. Indeed, $0\in\mathcal{U}_M$. Furthermore, as a consequence of the approximate controllability property of Equation (\ref{e-1-2}) for any fixed $T>0$ (see Theorem 1.4 in \cite{b2}), we have that the set $\mathcal{V}_T$ is also nonempty.
\par
    Now, for each admissible control $u\in\mathcal{U}_M$ of the infinite horizon control problem, we define a cost functional:
\begin{equation}\label{e-1-3}
    T(u)=\inf\{t>0;y(t;u,y_0)\in B(0,r)\}.
\end{equation}
    In this paper, the following two control problems are studied:
\vskip 5pt
    \noindent$(NP)_T$  $\;\;\inf_{v\in\mathcal{V}_T}\{\|v\|_{L^\infty(0,T;L^2(\Omega))}\}$;
\vskip 5pt
    \noindent$(TP)_M$  $\;\;\inf_{u\in\mathcal{U}_M}\{T(u)\}$.
\vskip 5pt
    \noindent The problem $(NP)_T$ is called minimal norm control problem (or optimal norm control problem) and the problem $(TP)_M$ is called minimal time control problem (or optimal time control problem).
    Following the symbols of \cite{b3}, we define the following two real value functions:
\begin{equation}\label{e-1-4}
    \alpha(T)\equiv\inf_{v\in\mathcal{V}_T}\{\|v\|_{L^\infty(0,T;L^2(\Omega))}\}\;\;\mbox{and}
    \;\;\tau(M)\equiv\inf_{u\in\mathcal{U}_M}\{T(u)\},
\end{equation}
    as the minimal (or optimal) norm and the minimal (or optimal) time for Problems $(NP)_T$ and $(TP)_M$, respectively. If a control $v^*_T\in L^\infty(0,T;L^2(\Omega))$ such that $y(T;v^*_T,y_0)\in B(0,r)$ and $\|v^*_T\|_{L^\infty(0,T;L^2(\Omega))}=\alpha(T)$, then it is called  the optimal norm control (or minimal norm control) to Problem $(NP)_T$. Similarly, if a control $u^*_M\in L^\infty(\mathbb{R}^+;L^2(\Omega))$ such that $y(\tau(M);u^*_M,y_0)\in B(0,r)$ and $\|u^*_M\|_{L^\infty(\mathbb{R}^+;L^2(\Omega))}\leq M$, then it is called  the optimal time control (or minimal time control) to Problem $(TP)_M$. In this paper, we let
$$
    \gamma(y_0)\equiv\inf\{t>0;y(t;0,y_0)\in B(0,r)\}.
$$
    By (\ref{p-1-2}), we know that $\gamma(y_0)<+\infty$ for any $y_0\in L^2(\Omega)$.
\par
    The main result of this paper can be presented as follows:
\begin{theorem}\label{main-theorem}
    Suppose that $(H_1)$ and $(H_2)$ hold. For each $T\in(0,\gamma(y_0)]$, the norm optimal control $v^*_T$ to Problem $(NP)_T$, when extended by zero to $(T,+\infty)$ is the optimal time control to $(TP)_{\alpha(T)}$. Conversely, for each $M\geq0$, the optimal time control $u^*_M$ to Problem $(TP)_M$, when restricted over $(0,\tau(M))$ is the optimal norm control to $(NP)_{\tau(M)}$.
\end{theorem}
\par
    The equivalence of minimal time and minimal norm control problems governed by infinite
    dimensional systems were found in many papers or books (see \cite{b7, b12, b13, b8, b3, b11}). In the case of the control acts globally into the controlled heat equation, i.e. $\omega=\Omega$, the related results were listed in \cite{b7}. Recently, when the control acts locally into the controlled heat equation, the same results were established in \cite{b8, b3}. These results are important to study the properties of minimal norm control problems, for instance, Bang-Bang property, explicit formula and the uniqueness of optimal controls to these problems. Besides, it also can be used to study the approximate property for perturbed time optimal control problems (see \cite{b9}).
\par
    The equivalence problem also appears for the controlled wave equation (see \cite{b12, b13}). However, in contrast to the heat equation, for the wave equation, as we know, the corresponding optimal control do not have the Bang-Bang property. Therefore, in general the time optimal controls are not uniquely determined (see Theorem 9.1 in \cite{b12}). Another difference between the two cases is that, in contrast to the result given in Theorem \ref{m-theorem1} of our paper, the value function corresponding to $\alpha(T)$ for the wave equation is in general not continuous (see \cite{b13}).
\par
    It is worth noting that all  controlled equations mentioned above are linear. In this paper, we shall establish the equivalence of minimal time and minimal norm control problems governed by semilinear heat equations. Moreover, the Bang-Bang property for minimal norm controls (see Corollary \ref{m-corollary1}) can be obtained in our paper. This property for the linear heat equation has proved by \cite{b8} in which the controlled equation is linear and the target set is the original point.
    As far as we know, no paper gets the Bang-Bang property of minimal norm controls for the semilinear controlled system. In contrast to the linear case, when the controlled equation is semilinear, we must exploit an abstract criterion to  show the compactness of the constructed sequences. For this, we can see the proofs of (\ref{p-1-21}) and (\ref{m-1-15}), and so on. On the other hand, in our case, we generally cannot deduce the uniqueness of optimal controls  to $(TP)_M$ and $(NP)_T$ from the Bang-Bang property.
\par
    The rest of the paper is organized as follows: Section 2 presents some necessary lemmas which play the important roles in our paper. In Section 3, we shall give the proof of Theorem \ref{main-theorem}. As a consequence of the main theorem, the Bang-Bang property for minimal norm controls are also presented in this section. The proof of (\ref{p-1-2}) will be given in the appendix.

\section{Preliminaries}
    In this section, we shall present some necessary lemmas for the proof of our main result.
    The first lemma concerns the existence of minimal time controls for Problem $(TP)_M$.
\begin{lemma}\label{p-lemma-2}
    Under the assumptions $(H_1)$ and $(H_2)$, for each $M>0$, the problem $(TP)_M$ has a control $u^*_M\in L^\infty(\mathbb{R}^+;L^2(\Omega))$ such that
\begin{equation}\label{p-1-7}
    \|u^*_M\|_{L^\infty(\mathbb{R}^+;L^2(\Omega))}\leq M
\end{equation}
    and
\begin{equation}\label{p-1-8}
    y(\tau(M);u^*_M,y_0)\in B(0,r).
\end{equation}
\end{lemma}
\begin{proof}
    First, we note that from (\ref{p-1-2}), the admissible control set $\mathcal{U}_M$ is nonempty. Indeed, $0\in\mathcal{U}_M$. We assume that $T_n\searrow \tau(M)$ and $y(T_n;u_n,y_0)\in B(0,r)$, where $u_n\in\mathcal{U}_M$. Without loss of generality, we assume that $T_n\in[\tau(M),\tau(M)+\eta]$, where $\eta>0$ is a constant.
    Since $\{u_n\}_{n\in\mathbb{N}}$ is bounded in $L^\infty(\mathbb{R}^+;L^2(\Omega))$, we can conclude that there exist a subsequence, still denoted in the same way, and a control $\tilde{u}\in\mathcal{U}_M$ such that
\begin{equation}\label{p-1-9}
    \chi_{(0,\tau(M)+\eta)}u_n\to \tilde{u}\;\;\mbox{weakly star in}\;\;L^\infty(\mathbb{R}^+;L^2(\Omega))\;\;\mbox{as}\;\;n\to\infty.
\end{equation}
\par
    Next, we shall show that there exists a subsequence of $\{u_n\}_{n\in\mathbb{N}}$, still denoted in the same way, such that
\begin{equation}\label{p-1-10}
    \|y(T_n;u_n,y_0)-y(\tau(M);\tilde{u},y_0)\|_{L^2(\Omega)}\to 0\;\;\mbox{as}\;\;n\to\infty.
\end{equation}
\par
    For this purpose, we first prove that
\begin{equation}\label{p-1-11}
    \|y(\tau(M);u_n,y_0)-y(\tau(M);\tilde{u},y_0)\|_{L^2(\Omega)}\to 0\;\;\mbox{as}\;\;n\to \infty.
\end{equation}
    For simplicity, we let $y_n(t)\equiv y(t;u_n,y_0)$ and $\tilde{y}(t)\equiv y(t;\tilde{u},y_0)$. Multiplying the equation (\ref{e-1-1}) by $y_n$, where $u$ is replaced by $u_n$, and integrating on $\Omega$, we get
\begin{equation}\label{p-1-12}
    \frac{1}{2}\frac{d}{dt}\|y_n(t)\|^2_{L^2(\Omega)}+\|\nabla y_n(t)\|^2_{L^2(\Omega)}+\langle f(y_n(t)),y_n(t)\rangle_{L^2(\Omega)}
    =\langle\chi_\omega u_n(t),y_n(t)\rangle_{L^2(\Omega)}.
\end{equation}
    From $(H_1)$, we have
\begin{equation}\label{p-1-13}
    \frac{1}{2}\frac{d}{dt}\|y_n(t)\|_{L^2(\Omega)}^2+\|\nabla y_n(t)\|_{L^2(\Omega)}^2\leq
    \frac{M}{4}+\|y_n(t)\|_{L^2(\Omega)}^2.
\end{equation}
 This means
\begin{equation}\label{p-1-14}
    \|y_n(t)\|_{L^2(\Omega)}^2\leq \|y_0\|_{L^2(\Omega)}^2+\frac{M}{2}(\tau(M)+\eta)
    +2\int_0^t\|y_n(t)\|_{L^2(\Omega)}^2dt
\end{equation}
    and
\begin{equation}\label{p-1-15}
    \int_0^{\tau(M)+\eta}\|\nabla y_n(t)\|^2_{L^2(\Omega)}dt\leq \frac{1}{2}\|y_0\|_{L^2(\Omega)}^2+\frac{M}{4}(\tau(M)+\eta)+\int_0^{\tau(M)+\eta}\|y_n(t)\|^2_{L^2(\Omega)}dt.
\end{equation}
    By the Gronwall inequality and (\ref{p-1-14}), we get
\begin{equation}\label{p-1-16}
    \sup_{t\in[0,\tau(M)+\eta]}\|y_n(t)\|_{L^2(\Omega)}\leq \left(\frac{M}{2}(\tau(M)+\eta)+\|y_0\|^2_{L^2(\Omega)}\right)^{\frac{1}{2}}
    e^{\tau(M)+\eta}
\end{equation}
    and
\begin{eqnarray}\label{p-1-17}
    &\;&\int_0^{\tau(M)+\eta}\|\nabla y_n(t)\|_{L^2(\Omega)}^2dt\nonumber\\
    &\leq&\left(e^{2(\tau(M)+\eta)}(\tau(M)+\eta)+\frac{1}{2}\right)\left(\frac{M}{2}(\tau(M)+\eta)+\|y_0\|_{L^2(\Omega)}^2\right).
\end{eqnarray}
    Hence
\begin{equation}\label{p-1-18}
    \{y_n\}_{n\in\mathbb{N}}\;\;\mbox{is bounded in}\;\;C([0,\tau(M)+\eta];L^2(\Omega))\cap
    L^2(0,\tau(M)+\eta;H_0^1(\Omega)).
\end{equation}
    However,
\begin{eqnarray}\label{p-1-19}
    &\;&\|(y_n)_t\|_{L^2(0,\tau(M)+\eta;H^{-1}(\Omega))}\nonumber\\
    &\leq&\|\triangle y_n\|_{L^2(0,\tau(M)+\eta;H^{-1}(\Omega))}+\|f(y_n)\|_{L^2(0,\tau(M)+\eta;H^{-1}(\Omega))}
    +\|\chi_\omega u_n\|_{L^2(0,\tau(M)+\eta;H^{-1}(\Omega))}\nonumber\\
    &\leq&(L+1)\|y_n\|_{L^2(0,\tau(M)+\eta;H_0^1(\Omega))}
    +M(\tau(M)+\eta)^{\frac{1}{2}}.
\end{eqnarray}
    Therefore,
\begin{equation}\label{p-1-20}
    \{(y_n)_t\}_{n\in\mathbb{N}}\;\;\mbox{is bounded in}\;\;L^2(0,\tau(M)+\eta;H^{-1}(\Omega)).
\end{equation}
    From Aubin's theorem (see Page 24 in \cite{b10}), there exists a subsequence of $\{y_n\}_{n\in\mathbb{N}}$, still denoted in the same way and $\bar{y}$ such that
\begin{equation}\label{p-1-21}
\begin{cases}
    y_n\to\bar{y}\;\;\mbox{strongly in}\;\; L^2(0,\tau(M)+\eta;L^2(\Omega))\;\;\mbox{as}\;\;n\to \infty,\\
    y_n\to\bar{y}\;\;\mbox{weakly in}\;\; L^2(0,\tau(M)+\eta;H_0^1(\Omega))\;\;\mbox{as}\;\;n\to \infty,\\
    (y_n)_t\to\bar{y}_t\;\;\mbox{weakly in}\;\; L^2(0,\tau(M)+\eta;H^{-1}(\Omega))\;\;\mbox{as}\;\;n\to \infty.
\end{cases}
\end{equation}
    Next, we show that $\bar{y}=\tilde{y}$. For this purpose, we only need to show the following identity holds:
\begin{equation}\label{p-1-22}
    \bar{y}_t-\triangle\bar{y}+f(\bar{y})=\chi_\omega\tilde{u}
    \;\;\mbox{in}\;\;L^2(0,\tau(M)+\eta;H^{-1}(\Omega)).
\end{equation}
    We first note that
\begin{eqnarray}\label{p-1-23}
    &\;&\int_0^{\tau(M)+\eta}\|f(y_n(t))-f(\bar{y}(t))\|_{L^2(\Omega)}^2dt\nonumber\\
    &\leq& L^2\int_0^{\tau(M)+\eta}\|y_n(t)-\bar{y}(t)\|_{L^2(\Omega)}^2dt\to 0\;\;\mbox{as}\;\;n\to\infty.
\end{eqnarray}
    From the definition of weak solution, we get that for any $\varphi\in L^2(0,\tau(M)+\eta;H_0^1(\Omega))$ and $n\in\mathbb{N}$,
\begin{equation}\label{p-1-24}
    \langle(y_n)_t-\triangle y_n+f(y_n)-\chi_\omega u_n,\varphi\rangle_{L^2(0,\tau(M)+\eta;H^{-1}(\Omega)),L^2(0,\tau(M)+\eta;
    H_0^1(\Omega))}=0.
\end{equation}
    By (\ref{p-1-21}) and (\ref{p-1-23}), letting $n\to\infty$ in the above identity, we have
\begin{equation}\label{p-1-25}
    \langle \bar{y}_t-\triangle \bar{y}+f(\bar{y})-\chi_\omega\tilde{u},\varphi
    \rangle_{L^2(0,\tau(M)+\eta;H^{-1}(\Omega)),L^2(0,\tau(M)+\eta;
    H_0^1(\Omega))}=0.
\end{equation}
    Hence $\bar{y}=\tilde{y}$. Let $z_n\equiv y_n-\tilde{y}$, then $z_n$ satisfies that
\begin{equation}\label{p-1-26}
\begin{cases}
    (z_n)_t-\triangle z_n+f(y_n)-f(\tilde{y})=\chi_\omega(u_n-\tilde{u})&\mbox{in}
    \;\;\Omega\times(0,\tau(M)+\eta),\\
    z_n=0&\mbox{on}\;\;\partial\Omega\times(0,\tau(M)+\eta),\\
    z_n(0)=0&\mbox{in}\;\;\Omega.
\end{cases}
\end{equation}
    Multiplying this equation by $z_n$ and integrating on $\Omega$, we have
\begin{eqnarray}\label{p-1-27}
    &\;&\frac{1}{2}\frac{d}{dt}\|z_n(t)\|^2_{L^2(\Omega)}+\|\nabla z_n(t)\|^2_{L^2(\Omega)}\nonumber\\
    &\leq&\langle f(\tilde{y}(t))-f(y_n(t)),z_n(t)\rangle_{L^2(\Omega)}
    +\langle\chi_\omega(u_n(t)-\tilde{u}(t)),z_n(t)\rangle_{L^2(\Omega)}.
\end{eqnarray}
    From (\ref{p-1-21}), this means that
\begin{eqnarray}\label{p-1-28}
    &\;&\frac{1}{2}\sup_{t\in{[0,\tau(M)+\eta]}}\|z_n(t)\|_{L^2(\Omega)}^2\nonumber\\
    &\leq& L\int_0^{\tau(M)+\eta}\|z_n(s)\|_{L^2(\Omega)}^2ds
    +\int_0^{\tau(M)+\eta}\|u_n(s)-\tilde{u}(s)\|_{L^2(\Omega)}
    \|z_n(s)\|_{L^2(\Omega)}ds\nonumber\\
    &\leq&L\int_0^{\tau(M)+\eta}\|z_n(s)\|_{L^2(\Omega)}^2ds
    +2M\int_0^{\tau(M)+\eta}\|z_n(s)\|_{L^2(\Omega)}ds\nonumber\\
    &\;&\to 0\;\;\mbox{as}\;\;n\to\infty.
\end{eqnarray}
    Hence
\begin{equation}\label{p-1-29}
    \|y_n-\tilde{y}\|_{C([0,\tau(M)+\eta];L^2(\Omega))}
    =\sup_{t\in[0,\tau(M)+\eta]}\|z_n(t)\|_{L^2(\Omega)}\to 0\;\;\mbox{as}\;\;n\to\infty.
\end{equation}
    This gives (\ref{p-1-11}).
\par
    Next, we prove that
\begin{equation}\label{p-1-30}
    \|y(T_n;u_n,y_0)-y(\tau(M);u_n,y_0)\|_{L^2(\Omega)}\to 0\;\;\mbox{as}\;\;n\to\infty.
\end{equation}
    Since $y(T_n;u_n,y_0)=y(T_n-\tau(M);u_n,y(\tau(M);u_n,y_0))$, we have
\begin{eqnarray}\label{p-1-31}
    &\;&y(T_n;u_n,y_0)=e^{\triangle(T_n-\tau(M))}y(\tau(M);u_n,y_0)\nonumber\\
    &\;&+\int_0^{T_n-\tau(M)}e^{\triangle(T_n-\tau(M)-t)}(-f(y(t;u_n,y_0))+\chi_\omega u_n(t))dt.
\end{eqnarray}
    This yields that
\begin{eqnarray}\label{p-1-32}
    &\;&\|y(T_n;u_n,y_0)-y(\tau(M);u_n,y_0)\|_{L^2(\Omega)}\nonumber\\
    &\leq&\left\|\left(e^{\triangle(T_n-\tau(M))}
    -1\right)y(\tau(M);u_n,y_0)\right\|_{L^2(\Omega)}\nonumber\\
    &\;&+\int_0^{T_n-\tau(M)}(\|f(y(t;u_n,y_0))\|_{L^2(\Omega)}
    +\|u_n(t)\|_{L^2(\Omega)})dt\nonumber\\
    &\leq&\left\|\left(e^{\triangle(T_n-\tau(M))}-1\right)(y(\tau(M);u_n,y_0)
    -y(\tau(M);\tilde{u},y_0))
    \right\|_{L^2(\Omega)}\nonumber\\
    &\;&+\left\|\left(e^{\triangle(T_n-\tau(M))}
    -1\right)y(\tau(M);\tilde{u},y_0)\right\|_{L^2(\Omega)}\nonumber\\
    &\;&+\int_0^{T_n-\tau(M)}\left(M
    +L\|y(t;\tilde{u},y_0)\|_{L^2(\Omega)}\right)dt\nonumber\\
    &\equiv&I_n^1+I_n^2+I_n^3.
\end{eqnarray}
    From (\ref{p-1-11}), it is clear that
\begin{equation}\label{p-1-33}
    I_n^1\leq 2\|y(\tau(M);u_n,y_0)-y(\tau(M);\tilde{u},y_0)\|_{L^2(\Omega)}\to 0\;\;\mbox{as}\;\;n\to\infty.
\end{equation}
    On the other hand, we note that $z(t)=e^{\triangle t}y(\tau(M);\tilde{u},y_0)$ is the solution of the following equation
\begin{equation}\label{yu-r-2-1}
\begin{cases}
    z_t-\triangle z=0&\mbox{in}\;\;\Omega\times(0,\delta),\\
    z=0&\mbox{on}\;\;\partial\Omega\times(0,\delta),\\
    z(0)=y(\tau(M);\tilde{u},y_0)&\mbox{in}\;\;\Omega.
\end{cases}
\end{equation}
    From the continuity of the solution for the equation (\ref{yu-r-2-1}), we get that
\begin{equation}\label{yu-r-2-2}
    I_n^2=\|z(T_n-\tau(M))-z(0)\|_{L^2(\Omega)}\to 0\;\;\mbox{as}\;\;n\to\infty.
\end{equation}
    Similar to the proof of (\ref{p-1-16}), there exists a constant $C>0$ such that
$$
    \sup_{t\in[0,\tau(M)+\eta]}\|y(t;\tilde{u},y_0)\|_{L^2(\Omega)}\leq C.
$$
    This implies that
\begin{equation}\label{p-1-35}
    I_n^3\leq(M+L C)(T_n-\tau(M))\to 0\;\;\mbox{as}\;\;n\to\infty.
\end{equation}
    Therefore, from (\ref{p-1-33}), (\ref{yu-r-2-2}) and (\ref{p-1-35}), we get
\begin{equation}\label{p-1-36}
    \|y(T_n;u_n,y_0)-y(\tau(M);u_n,y_0)\|_{L^2(\Omega)}\to 0\;\;\mbox{as}\;\;n\to\infty.
\end{equation}
    This, together with (\ref{p-1-11}), yields (\ref{p-1-10}).
    By (\ref{p-1-10}) and the fact of $y(T_n;u_n,y_0)\in B(0,r)$, we have
\begin{equation}\label{p-1-37}
    y(\tau(M);\tilde{u},y_0)\in B(0,r).
\end{equation}
    Let
$$
    u^*_M=
\begin{cases}
    \tilde{u}&\mbox{in}\;[0,\tau(M)),\\
    0&\mbox{in}\;[\tau(M),+\infty),
\end{cases}
    $$
    we have (\ref{p-1-7}) and (\ref{p-1-8}). The proof is completed.
\end{proof}
    Next, we prove the existence of  a solution of the problem $(NP)_T$.
\begin{lemma}\label{p-lemma-3}
    Suppose that $(H_1)$ and $(H_2)$ hold. Then, for each $T>0$, the problem $(NP)_T$ has a control $v^*_T\in L^\infty(0,T;L^2(\Omega))$ such that
\begin{equation}\label{p-1-37}
    \|v^*_T\|_{L^\infty(0,T;L^2(\Omega))}\leq\alpha(T)
\end{equation}
    and
\begin{equation}\label{p-1-38}
    y(T;v^*_T,y_0)\in B(0,r).
\end{equation}
\end{lemma}
\begin{proof}
    By the approximate controllability of Equation (\ref{e-1-1}) (see Theorem 1.4 in \cite{b2}), we know that for any $T>0$, there exists a control $v\in L^\infty(0,T;L^2(\Omega))$ such that
\begin{equation}\label{p-1-39}
    y(T;v,y_0)\in B(0,r).
\end{equation}
    Then the set $\mathcal{V}_T$ is nonempty. Let $\{v_n\}_{n\in\mathbb{N}}\subset \mathcal{V}_T$ such that
\begin{equation}\label{p-1-40}
    \|v_n\|_{L^\infty(0,T;L^2(\Omega))}\to\alpha(T)\;\;\mbox{as}\;\;n\to\infty.
\end{equation}
    It follows that $\{v_n\}_{n\in\mathbb{N}}$ is bounded in $L^\infty(0,T;L^2(\Omega))$. Therefore, there exists a subsequence of $\{v_n\}_{n\in\mathbb{N}}$, still denoted in the same way, and $\bar{v}\in L^\infty(0,T;L^2(\Omega))$ such that
\begin{equation}\label{p-1-41}
    v_n\to \bar{v}\;\;\mbox{weakly star in}\;\; L^\infty(0,T;L^2(\Omega))\;\;\mbox{as}\;\;n\to\infty.
\end{equation}
    Similar to the proof of (\ref{p-1-11}) in Lemma \ref{p-lemma-2}, we have that, there exists a subsequence of $\{v_n\}_{n\in\mathbb{N}}$, still denoted in the same way, such that
\begin{equation}\label{p-1-42}
    \|y(T;v_n,y_0)-y(T;\bar{v},y_0)\|_{L^2(\Omega)}\to 0\;\;\mbox{as}\;\;n\to\infty.
\end{equation}
    This implies that
\begin{equation}\label{p-1-43}
    y(T;\bar{v},y_0)\in B(0,r).
\end{equation}
    Then  $\bar{v}\in \mathcal{V}_T$. From the weakly star lower semi-continuity of $L^\infty$-norm and (\ref{p-1-40}), we get
\begin{equation}\label{p-1-44}
    \|\bar{v}\|_{L^\infty(0,T;L^2(\Omega))}\leq \liminf_{n\to\infty}\|v_n\|_{L^\infty(0,T;L^2(\Omega))}=\alpha(T).
\end{equation}
    Let $v^*_T=\bar{v}$, from (\ref{p-1-43}) and (\ref{p-1-44}), we complete the proof.
\end{proof}
    The following proposition contains the maximum principle for the problem $(TP)_M$. The proof of this proposition has been presented in Theorem 4.1 of Chapter 7 in \cite{b1} (see also the proof of Theorem 1 in \cite{b6}). Then, we omit it in our paper.
\begin{proposition}\label{p-proposition-1}
    Assume that $(H_1)$ and $(H_2)$ hold. For each $M>0$, let $\tau(M)$ be the optimal time and $u^*_M$ be the optimal control for the problem $(TP)_M$. Then there exists $(\xi,\psi)\in L^2(\Omega)\times C([0,\tau(M)];L^2(\Omega))$ satisfied
\begin{equation}\label{p-proposition1-1}
    \xi\neq 0
\end{equation}
    and
\begin{equation}\label{p-proposition1-2}
\begin{cases}
    \psi_t+\triangle\psi-f'(y(\cdot;u^*,y_0))\psi=0&\mbox{in}\;\;\Omega\times(0,\tau(M)),\\
    \psi=0&\mbox{on}\;\;\partial\Omega\times(0,\tau(M)),\\
    \psi(\tau(M))=\xi&\mbox{in}\;\;\Omega
\end{cases}
\end{equation}
    such that
\begin{equation}\label{p-proposition1-3}
    \langle\chi_\omega\psi(t),u^*_M(t)\rangle_{L^2(\Omega)}=\max_{\|v\|_{L^2(\Omega)}\leq M}\langle
    \chi_\omega\psi(t),v\rangle_{L^2(\Omega)}\;\;\mbox{for almost all}\;\;t\in(0,\tau(M)).
\end{equation}
\end{proposition}
    From this proposition, we can deduce the Bang-Bang property for the optimal control of the problem $(TP)_M$.
\begin{corollary}\label{p-corollary-1}
    Assume that $(H_1)$ and $(H_2)$ hold. For each $M>0$, if $\tau(M)$ and $u^*_M$ are the optimal time and optimal control to the problem $(TP)_M$, respectively. Then
\begin{equation}\label{p-corollary1-1}
    \|u^*_M(t)\|_{L^2(\Omega)}=M\;\;\mbox{for almost all}\;\;t\in(0,\tau(M)).
\end{equation}
\end{corollary}
\begin{proof}
    Before to prove (\ref{p-corollary1-1}), we first recall that if $\xi\neq 0$, then
\begin{equation}\label{p-2-63}
    \|\chi_\omega\psi(t)\|_{L^2(\Omega)}\neq 0\;\;\mbox{for all}\;\;t\in(0,\tau(M)),
\end{equation}
    where $\psi$ is the unique solution to Equation (\ref{p-proposition1-2}). When the domain $\Omega$ is a convex subset of $\mathbb{R}^N$, this result which is called the property of unique continuation for the semilinear heat equation was proved in \cite{b4} (see Proposition 2.1 in \cite{b4}). Recently, in \cite{b5}, the authors have proved that, indeed, the assumption of convexity for $\Omega$ can be removed (see Theorem 4 in \cite{b5}). In fact, if (\ref{p-2-63}) does not hold, then by the unique continuation for the semilinear heat equation, we know that $\psi(\tau(M))=0$. It contradicts (\ref{p-proposition1-1}). Hence (\ref{p-2-63}) holds.
\par
    Now, we shall prove (\ref{p-corollary1-1}). Suppose that (\ref{p-corollary1-1}) did not hold. Then there exists a measurable set $e\subset(0,\tau(M))$ with $\mbox{mes}(e)>0$ such that
\begin{equation}\label{p-2-64}
    \|u^*_M(t)\|_{L^2(\Omega)}<M\;\;\mbox{for all}\;\;t\in e,
\end{equation}
    where $\mbox{mes}(A)$ is the Lebesgue measure of $A\subset \mathbb{R}$. However, from (\ref{p-proposition1-3}), we have that
\begin{equation}\label{p-2-65}
    M\|\chi_\omega \psi(t)\|_{L^2(\Omega)}=\langle \chi_\omega\psi(t),u^*_M(t)\rangle_{L^2(\Omega)}\;\;\mbox{for almost all}\;\;t\in(0,\tau(M)).
\end{equation}
    This gives that
\begin{equation}\label{p-2-66}
    M\|\chi_\omega \psi(t)\|_{L^2(\Omega)}\leq\| \chi_\omega\psi(t)\|_{L^2(\Omega)}\|u^*_M(t)\|_{L^2(\Omega)}\;\;\mbox{for almost all}\;\;t\in(0,\tau(M)).
\end{equation}
    Then, by (\ref{p-2-63}), we get
\begin{equation}\label{p-2-67}
    \|u^*_M(t)\|\geq M\;\;\mbox{for almost all}\;\;t\in e.
\end{equation}
    It contradicts the assumption (\ref{p-2-64}). The proof is completed.
\end{proof}
\section{The proof of the main result}
    In this section, we shall present the proof of Theorem \ref{main-theorem}. For this purpose, we first prove the following theorem (i.e., Theorem \ref{m-theorem1}). This theorem states that
    the minimal time as a function of the norm bound $M$ and the minimal control norm as a function of $T$ are continuous, and inverse to each other.
\begin{theorem}\label{m-theorem1}
    Let $\gamma(y_0)=\inf\{t>0;y(t;0,y_0)\in B(0,r)\}$. Then, under the assumptions $(H_1)$ and $(H_2)$, the function $\tau(\cdot)$ is strictly monotonically decreasing and continuous from $[0,+\infty)$ onto $(0,\gamma(y_0)]$. Moreover, it holds that
\begin{equation}\label{m-theorem-1}
    \tau(\alpha(T))=T\;\;\mbox{for all}\;\;T\in(0,\gamma(y_0)]\;\;\mbox{and}\;\;
    \alpha(\tau(M))=M\;\;\mbox{for all}\;\;M\in[0,+\infty).
\end{equation}
    Consequently, the maps $M\to\tau(M)$ and $T\to\alpha(T)$ are the inverse of each other.
\end{theorem}

\begin{proof}
    We follow the idea of \cite{b3}. The proof shall be divided into several steps as follows:
\vskip 8pt
    \emph{Step 1. The function $\tau(\cdot)$ is strictly monotonically decreasing over $[0,+\infty)$.}
\vskip 5pt
    Let $M_1>M_2\geq0$, we shall show that $\tau(M_1)<\tau(M_2)$. Suppose that $\tau(M_1)\geq \tau(M_2)$.  We will find a contradiction. From (\ref{p-1-2}) and Lemma \ref{p-lemma-2}, we know that there exists a time optimal control $u^*_{M_2}$ to Problem $(TP)_{M_2}$ such that
\begin{equation}\label{m-1-1}
    \|\chi_{(0,\tau(M_2))}u^*_{M_2}\|_{L^\infty(\mathbb{R}^+;L^2(\Omega))}\leq M_2<M_1
\end{equation}
    and
\begin{eqnarray}\label{m-1-2}
    &\;&\|y(\tau(M_1);\chi_{(0,\tau(M_2))}u_{M_2}^*,y_0)\|_{L^2(\Omega)}\nonumber\\
    &=&\|y(\tau(M_1)-\tau(M_2);0,y(\tau(M_2);u_{M_2}^*,y_0))\|_{L^2(\Omega)}\nonumber\\
    &\leq&e^{-\lambda_1(\tau(M_1)-\tau(M_2))}\|y(\tau(M_2);u_{M_2}^*,y_0)\|_{L^2(\Omega)}\leq r.
\end{eqnarray}
    These imply that $\chi_{(0,\tau(M_2))}u^*_{M_2}$ is an optimal control to Problem $(TP)_{M_1}$. By Corollary \ref{p-corollary-1} (the Bang-Bang property for the optimal control to Problem $(TP)_{M_1}$), we have
\begin{equation}\label{m-1-3}
    \|\chi_{(0,\tau(M_2))}(t)u^*_{M_2}(t)\|_{L^2(\Omega)}=M_1\;\;\mbox{over}\;\;(0,\tau(M_1)).
\end{equation}
    It contradicts to (\ref{m-1-1}). Then $\tau(M_1)<\tau(M_2)$.
\vskip 8pt
    \emph{Step 2. When $\{M_n\}_{n\in\mathbb{N}}$ is such that $M_1\geq M_2\geq \cdots \geq M_n\to M\in [0,+\infty)$ as $n\to\infty$, then $\lim_{n\to \infty}\tau(M_n)=\tau(M)$. Consequently, the function $\tau(\cdot)$ is right-continuous in $\mathbb{R}^+$.}
\vskip 5pt
    First, we note that from the conclusion of Step 1,
\begin{equation}\label{m-1-4}
    \tau(M_1)\leq\tau(M_2)\leq\cdots\leq\tau(M_n)\leq\cdots\leq \tau(M).
\end{equation}
    If
\begin{equation}\label{m-1-5}
    \lim_{n\to\infty}\tau(M_n)\neq\tau(M),
\end{equation}
     we have
\begin{equation}\label{m-1-6}
    \tau(M_n)\nearrow(\tau(M)-\delta)\;\;\mbox{for some}\;\;\delta>0\;\;\mbox{as}\;\;n\to\infty.
\end{equation}
    On the other hand, since $u^*_{M_n}$ is the optimal control to Problem $(TP)_{M_n}$, $n\in\mathbb{N}$, we can conclude that
\begin{equation}\label{m-1-7}
    y(\tau(M_n);u_{M_n}^*,y_0)\in B(0,r)
\end{equation}
    and
\begin{equation}\label{m-1-8}
    \|u^*_{M_n}\|_{L^\infty(\mathbb{R}^+;L^2(\Omega))}\leq M_n\leq M_1\;\;\mbox{for all}\;\;n.
\end{equation}
    Thus, there is a subsequence of $\{u_{M_n}^*\}_{n\in\mathbb{N}}$, still denoted in the same way, and $\bar{u}\in L^\infty(\mathbb{R}^+;L^2(\Omega))$ such that
\begin{equation}\label{m-1-9}
    u_{M_n}^*\to\bar{u}\;\;\mbox{weakly star in}\;\;L^\infty(\mathbb{R}^+;L^2(\Omega))\;\;\mbox{as}\;\;n\to\infty.
\end{equation}
    Now, we show that there exists a subsequence of $\{u^*_{M_n}\}_{n\in\mathbb{N}}$, still denoted in the same way, such that
\begin{equation}\label{m-1-10}
    \|y(\tau(M_n);u^*_{M_n},y_0)-y(\tau(M)
    -\delta;\bar{u},y_0)\|_{L^2(\Omega)}\to0\;\;\mbox{as}\;\;n\to\infty.
\end{equation}
    For this purpose, we only need to prove that
\begin{equation}\label{m-1-11}
    \|y(\tau(M_n);u^*_{M_n},y_0)-y(\tau(M_n);\bar{u},y_0)\|_{L^2(\Omega)}\to 0\;\;\mbox{as}\;\;n\to \infty
\end{equation}
    and
\begin{equation}\label{m-1-12}
    \|y(\tau(M_n);\bar{u},y_0)-y(\tau(M)-\delta;\bar{u},y_0)\|_{L^2(\Omega)}\to 0\;\;\mbox{as}\;\;n\to \infty.
\end{equation}
\par
    First, we show that (\ref{m-1-11}) holds.
    Similar to the proofs of (\ref{p-1-18}) and (\ref{p-1-20}), we have
\begin{equation}\label{m-1-13}
    \{y(\cdot;u_{M_n}^*,y_0)\}_{n\in\mathbb{N}}\;\;\mbox{is bounded in}\;\;
    C([0,\tau(M)-\delta];L^2(\Omega))\cap L^2(0,\tau(M)-\delta;H_0^1(\Omega))
\end{equation}
    and
\begin{equation}\label{m-1-14}
    \{y_t(\cdot;u^*_{M_n},y_0)\}_{n\in\mathbb{N}}\;\;\mbox{is bounded in}\;\;L^2(0,\tau(M)-\delta;H^{-1}(\Omega)).
\end{equation}
    Therefore, from Aubin's theorem, there exist a subsequence of $\{y(\cdot;u_{M_n}^*,y_0)\}_{n\in\mathbb{N}}$, still denoted in the same way, and $\tilde{y}$ such that
\begin{equation}\label{m-1-15}
\begin{cases}
    y(\cdot;u^*_{M_n},y_0)\to\tilde{y}\;\;\mbox{strongly in}\;\;L^2(0,\tau(M)-\delta;L^2(\Omega))\;\;\mbox{as}\;\;n\to\infty,\\
    y(\cdot;u^*_{M_n},y_0)\to\tilde{y}\;\;\mbox{weakly in}\;\;L^2(0,\tau(M)-\delta;H_0^1(\Omega))\;\;\mbox{as}\;\;n\to\infty,\\
    y_t(\cdot;u^*_{M_n},y_0)\to\tilde{y}_t\;\;\mbox{weakly in}\;\;L^2(0,\tau(M)-\delta;H^{-1}(\Omega))\;\;
    \mbox{as}\;\;n\to\infty.
\end{cases}
\end{equation}
    Now, we show that $\tilde{y}(\cdot)\equiv y(\cdot;\bar{u},y_0)$. For this purpose, we only need to prove that
\begin{equation}\label{m-1-16}
    \tilde{y}_t-\triangle\tilde{y}+f(\tilde{y})=\chi_\omega\bar{u}\;\;\mbox{in}
    \;\;L^2(0,\tau(M)-\delta;H^{-1}(\Omega)).
\end{equation}
    First, we note from (\ref{m-1-15}) that
\begin{eqnarray}\label{m-1-17}
    &\;&\int_0^{\tau(M)-\delta}\|f(y(t;u^*_{M_n},y_0))-f(\tilde{y}(t))\|_{L^2(\Omega)}^2dt
    \nonumber\\
    &\leq& L^2\int_0^{\tau(M)-\delta}\|y(t;u^*_{M_n},y_0)-\tilde{y}(t)\|_{L^2(\Omega)}^2dt
    \to 0\;\;\mbox{as}\;\;n\to\infty.
\end{eqnarray}
    Then, for any $\varphi\in L^2(0,\tau(M)-\delta;H_0^1(\Omega))$, from (\ref{m-1-15}), we have
\begin{eqnarray}\label{m-1-18}
    0&=&\langle y_t(\cdot;u^*_{M_n},y_0)-\triangle y(\cdot;u^*_{M_n},y_0)
    +f(y(\cdot;u^*_{M_n},y_0))\nonumber\\
    &\;&-\chi_\omega u_{M_n}^*,\varphi\rangle_{L^2(0,\tau(M)-\delta;H^{-1}(\Omega)),
    L^2(0,\tau(M)-\delta;H_0^1(\Omega))}\nonumber\\
    &\to&\langle \tilde{y}_t-\triangle\tilde{y}+f(\tilde{y})-\chi_\omega\bar{u},\varphi\rangle
    _{L^2(0,\tau(M)-\delta;H^{-1}(\Omega)),
    L^2(0,\tau(M)-\delta;H_0^1(\Omega))}\;\;\mbox{as}\;\;n\to\infty.
\end{eqnarray}
    This gives (\ref{m-1-16}). Hence
\begin{equation}\label{m-1-19}
    \tilde{y}(\cdot)\equiv y(\cdot;\bar{u},y_0).
\end{equation}
    Let $z_n(t)\equiv y(t;u_{M_n}^*,y_0)-y(t;\bar{u},y_0)$ for all $t\in(0,\tau(M)-\delta)$, then $z_n(\cdot)$ satisfies that
\begin{equation}\label{m-1-20}
\begin{cases}
    (z_n)_t-\triangle z_n+f(y(\cdot;u_{M_n}^*,y_0))\\
    ~~~~~~~-f(y(\cdot;\bar{u},y_0))
    =\chi_\omega(u^*_{M_n}-\bar{u})&\mbox{in}\;\;\Omega\times(0,\tau(M)-\delta),\\
    z_n=0&\mbox{on}\;\;\partial\Omega\times(0,\tau(M)-\delta),\\
    z_n(0)=0&\mbox{in}\;\;\Omega.
\end{cases}
\end{equation}
    Multiplying this equation by $z_n$ and integrating on $\Omega$, we get
\begin{eqnarray}\label{m-1-21}
    &\;&\frac{1}{2}\frac{d}{dt}\|z_n(t)\|_{L^2(\Omega)}^2+\|\nabla z_n(t)\|_{L^2(\Omega)}^2\nonumber\\
    &=&\langle f(t;\bar{u},y_0)-f(y(t;u^*_{M_n},y_0)),z_n(t)\rangle_{L^2(\Omega)}+\langle\chi_\omega(u^*_{M_n}(t)
    -\bar{u}(t)),z_n(t)\rangle_{L^2(\Omega)}.
\end{eqnarray}
    This, together with $(H_1)$, gives
\begin{eqnarray}\label{m-1-22}
    &\;&\frac{1}{2}\|z_n(t)\|_{L^2(\Omega)}^2+\int_0^t\|\nabla z_n(s)\|_{L^2(\Omega)}^2ds
    \nonumber\\&\leq&
    L\int_0^t\|z_n(s)\|_{L^2(\Omega)}^2ds+\left(\int_0^t\|u^*_{M_n}(s)-\bar{u}(s)\|^2
    _{L^2(\Omega)}ds\right)^{\frac{1}{2}}\left(\int_0^t\|z_n(s)\|^2_{L^2(\Omega)}ds\right)
    ^{\frac{1}{2}}.
\end{eqnarray}
    This, together with (\ref{m-1-15}) and (\ref{m-1-19}), implies
\begin{equation}\label{m-1-23}
    \sup_{t\in[0,\tau(M)-\delta]}\|z_n(t)\|_{L^2(\Omega)}\to 0\;\;\mbox{as}\;\;n\to\infty.
\end{equation}
    From the definition of $z_n$, we get (\ref{m-1-11}).
\par
    On the other hand, from the strong continuity of $y(\cdot;\bar{u},y_0)$ in $L^2(\Omega)$ and (\ref{m-1-6}), (\ref{m-1-12}) is obvious. Hence, together with (\ref{m-1-11}) and (\ref{m-1-12}), (\ref{m-1-10}) holds.
\par
    Since $y(\tau(M_n);u_{M_n}^*,y_0)\in B(0,r)$, from (\ref{m-1-10}), we can conclude that
\begin{equation}\label{m-1-24}
    y(\tau(M)-\delta;\bar{u},y_0)\in B(0,r).
\end{equation}
    On the other hand, from (\ref{m-1-9}), we have
\begin{equation}\label{m-1-25}
    \|\bar{u}\|_{L^\infty(\mathbb{R}^+;L^2(\Omega))}\leq\liminf_{n\to\infty}\|
    u^*_{M_n}\|_{L^\infty(\mathbb{R}^+;L^2(\Omega))}\leq \liminf_{n\to\infty}M_n=M.
\end{equation}
    This implies  $\bar{u}\in\mathcal{U}_M$.
    By (\ref{m-1-24}) and (\ref{m-1-25}), we get a contradiction to the optimality of $\tau(M)$ to Problem $(TP)_M$. Hence $\lim_{n\to\infty}\tau(M_n)=\tau(M)$.
\vskip 8pt
    \emph{Step 3. When $\{M_n\}_{n\in\mathbb{N}}$ is such that $M_1\leq M_2\leq \cdots \leq M_n\to M\in[0,+\infty)$ as $n\to\infty$, it holds that $\lim_{n\to\infty}\tau(M_n)=\tau(M)$. Consequently, the function $\tau(\cdot)$ is left-continuous.}
\vskip 5pt
    From the monotonicity of the function $\tau(\cdot)$ (see the conclusion of Step 1), we have
\begin{equation}\label{m-1-26}
    \tau(M_1)\geq \tau(M_2)\geq\cdots\geq\tau(M_n)\geq\cdots\geq\tau(M).
\end{equation}
    If $\lim_{n\to\infty}\tau(M_n)\neq\tau(M)$, then
\begin{equation}\label{m-1-27}
    \tau(M_n)\searrow(\tau(M)+\delta)\;\;\mbox{for some}\;\;\delta>0\;\;\mbox{as}\;\;n\to\infty
\end{equation}
    and
\begin{equation}\label{m-1-28}
    \tau(M_n)>(\tau(M)+\delta)\;\;\mbox{for all}\;\;n\in\mathbb{N}.
\end{equation}
    Since $\tau(M)$ and $u_M^*$ are the optimal time and optimal control to Problem $(TP)_M$, respectively. Let $y_M(\cdot)\equiv y(\cdot;\chi_{(0,\tau(M))}u^*_M,y_0)$ and $y_n(\cdot)\equiv y(\cdot;\frac{M_n}{M}\chi_{(0,\tau(M))}u^*_M,y_0)$. By the optimality of $u^*_M$ and $\tau(M)$ to Problem $(TP)_M$, we have
\begin{equation}\label{m-1-29}
    y_M(\tau(M))\in B(0,r),\;\;\mbox{i.e.}\;\;\|y_M(\tau(M))\|_{L^2(\Omega)}\leq r.
\end{equation}
    Let $z_n(t)\equiv y_M(t)-y_n(t)$ for each $t\in\mathbb{R}^+$. Then $z_n(\cdot)$ satisfies that
\begin{equation}\label{m-1-30}
\begin{cases}
    (z_n)_t-\triangle z_n+f(y_M)-f(y_n)=\left(1-\frac{M_n}{M}\right)\chi_\omega
    \chi_{(0,\tau(M))}u^*_M&\mbox{in}\;\;
    \Omega\times\mathbb{R}^+,\\
    z_n=0&\mbox{on}\;\;\partial\Omega\times\mathbb{R}^+,\\
    z_n(0)=0&\mbox{in}\;\;\Omega.
\end{cases}
\end{equation}
    Similar to the proof of (\ref{m-1-21}), we have
\begin{eqnarray}\label{m-1-31}
    &\;&\frac{1}{2}\frac{d}{dt}\|z_n(t)\|_{L^2(\Omega)}^2+\|\nabla z_n(t)\|_{L^2(\Omega)}^2+\langle f(y_M(t))-f(y_n(t)),z_n(t)\rangle_{L^2(\Omega)}\nonumber\\
    &=&\left\langle\left(1-\frac{M_n}{M}\right)
    \chi_\omega\chi_{(0,\tau(M))}u^*_M(t),z_n(t)\right\rangle_{L^2(\Omega)}.
\end{eqnarray}
    From this identity and $(H_1)$ with $L>0$, we have the following integral inequality
\begin{equation}\label{m-1-32}
    \|z_n(t)\|_{L^2(\Omega)}^2\leq (2L+1)\int_0^t\|z_n(t)\|_{L^2(\Omega)}^2ds
    +\left(1-\frac{M_n}{M}\right)^2\int_0^t\chi_{(0,\tau(M))}\|u^*_M(s)\|_{L^2(\Omega)}^2ds.
\end{equation}
    This, together with Gronwall's inequality, implies that
\begin{equation}\label{m-1-33}
    \sup_{t\in[0,\tau(M_1)]}\|z_n(t)\|_{L^2(\Omega)}\leq
    \left(1-\frac{M_n}{M}\right)M\sqrt{\tau(M)}
    e^{\frac{(2L+1)\tau(M_1)}{2}}.
\end{equation}
    However, from (\ref{p-1-2}), we have
\begin{equation}\label{m-1-34}
    \|y_M(t)\|_{L^2(\Omega)}\leq e^{-\lambda_1(t-\tau(M))}\|y_M(\tau(M))\|_{L^2(\Omega)}\;\;\mbox{for all}\;\;t\geq\tau(M).
\end{equation}
    This gives that
\begin{eqnarray}\label{m-1-35}
    &\;&\|y_n(\tau(M)+\delta)\|_{L^2(\Omega)}\nonumber\\
    &\leq&\|y_M(\tau(M)+\delta)
    -y_n(\tau(M)+\delta)\|_{L^2(\Omega)}+\|y_M(\tau(M)+\delta)\|_{L^2(\Omega)}\nonumber\\
    &=&\|y_M(\tau(M)+\delta)-y_n(\tau(M)+\delta)\|_{L^2(\Omega)}+\|y(\delta;0,y_M(\tau(M)))\|
    _{L^2(\Omega)}\nonumber\\
    &\leq&\left(1-\frac{M_n}{M}\right)M\sqrt{\tau(M)}e^{\frac{(2L+1)\tau(M_1)}{2}}
    +e^{-\lambda_1\delta}r.
\end{eqnarray}
    Since $M_n\nearrow M$, i.e. $\frac{M_n}{M}\nearrow 1$ as $n\to\infty$, we can find an $n_0\in\mathbb{N}$ such that
\begin{equation}\label{m-1-36}
    \frac{M_n}{M}\geq 1-\frac{(1-e^{-\lambda_1\delta})r}{M\sqrt{\tau(M)}}e^{-\frac{
    (2L+1)\tau(M_1)}{2}}\;\;\mbox{for all}\;\;n\geq n_0.
\end{equation}
    Hence, from (\ref{m-1-35}), we get
\begin{equation}\label{m-1-37}
    \|y_n(\tau(M)+\delta)\|_{L^2(\Omega)}\leq r,\;\;\mbox{i.e.}\;\;y_n(\tau(M)+\delta)\in B(0,r)\;\;\mbox{for all}\;\;n\geq n_0.
\end{equation}
    This implies
\begin{equation}\label{m-1-38}
    \tau(M_n)\leq\tau(M)+\delta\;\;\mbox{for all}\;\;n\geq n_0.
\end{equation}
    It contradicts to (\ref{m-1-28}). This gives the conclusion of Step 3.
\vskip 8pt
    \emph{Step 4. $\lim_{M\to0}\tau(M)=\gamma(y_0)$.}
\vskip 5pt
    If it did not hold, then there exists a subsequence $\{M_n\}_{n\in\mathbb{N}}$ with $M_1\geq M_2\geq \cdots\geq M_n\to 0$ as $n\to \infty$, such that
\begin{equation}\label{m-1-39}
    \lim_{n\to\infty}\tau(M_n)=T<\gamma(y_0).
\end{equation}
    From the optimality of $u^*_{M_n}$ to the problem $(TP)_{M_n}$, we have
\begin{equation}\label{m-1-40}
    \|u_{M_n}^*\|_{L^\infty(\mathbb{R}^+;L^2(\Omega))}\leq M_n\leq M_1\;\;\mbox{for all}\;\;n\in\mathbb{N}.
\end{equation}
    Thus, there exist a subsequence of $\{u_{M_n}^*\}_{n\in\mathbb{N}}$, still denoted  in the same way, and $\tilde{u}$ such that
\begin{equation}\label{m-1-41}
    u_{M_n}^*\to \tilde{u}\;\;\mbox{weakly star in}\;\;L^\infty(\mathbb{R}^+;L^2(\Omega))\;\;\mbox{as}\;\;n\to\infty.
\end{equation}
    It follows that
\begin{equation}\label{m-1-42}
    \|\tilde{u}\|_{L^\infty(\mathbb{R}^+;L^2(\Omega))}\leq\liminf_{n\to\infty}\|u_{M_n}^*\|
    _{L^\infty(\mathbb{R}^+;L^2(\Omega))}\leq\liminf_{n\to\infty}M_n=0.
\end{equation}
    This gives $\tilde{u}=0$. On the other hand, similar to the proof of (\ref{m-1-12}), we can deduce that
\begin{equation}\label{m-1-43}
    \|y(\tau(M_n);u_{M_n}^*,y_0)-y(T;\tilde{u},y_0)\|_{L^2(\Omega)}\to 0\;\;\mbox{as}\;\;n\to\infty.
\end{equation}
    Since $y(\tau(M_n);u^*_{M_n},y_0)\in B(0,r)$, (\ref{m-1-43}) follows that
\begin{equation}\label{m-1-44}
    y(T;\tilde{u},y_0)=y(T;0,y_0)\in B(0,r).
\end{equation}
    It contradicts  the definition of $\gamma(y_0)$. Then $\lim_{n\to\infty}\tau(M_n)=\gamma(y_0)$.
\vskip 8pt
    \emph{Step 5. $\lim_{M\to\infty}\tau(M)=0$.}
\vskip 5pt
    If it is false, then there exist a sequence $\{M_n\}$ with $M_1\leq M_2\leq\cdots\leq M_n\to \infty$ as $n\to\infty$, such that
\begin{equation}\label{m-1-45}
    \lim_{n\to\infty}\tau(M_n)=2T>0.
\end{equation}
    From the monotonicity of the function $\tau(\cdot)$, the following relation holds:
\begin{equation}\label{m-1-46}
    \tau(M_n)>T>0\;\;\mbox{for}\;\;n\in\mathbb{N}.
\end{equation}
    By the approximate controllability for semilinear heat equations (see Theorem 1.4 in
    \cite{b2}), there  is a control $u\in L^\infty(\mathbb{R}^+;L^2(\Omega))$ independent of $n$ such that $y(T;u,y_0)\in B(0,r)$. Since $M_n\to \infty$ as $n\to\infty$, we can find a $n_0\in\mathbb{N}$ such that
\begin{equation}\label{m-1-47}
    \|u\|_{L^\infty(\mathbb{R}^+;L^2(\Omega))}\leq M_n\;\;\mbox{for all}\;\;n\geq n_0.
\end{equation}
    Hence $u\in\mathcal{U}_{M_n}$. Then by the optimality if $\tau(M_n)$ for $(TP)_{M_n}$, we have $\tau(M_n)\leq T$. It is a contradiction to (\ref{m-1-46}). It follows that the result of $\lim_{M\to \infty}\tau(M)=0$ holds.

\vskip 8pt
    \emph{Step 6. The proof of (\ref{m-theorem-1}).}
\vskip 5pt
    From the conclusions of Step 2 and Step 3, we can conclude that $\tau(\cdot)$ is continuous over $\mathbb{R}^+$. We now prove the first identity in (\ref{m-theorem-1}). Fix a $T\in(0,\gamma(y_0)]$. By Lemma \ref{p-lemma-3}, we know that Problem $(NP)_T$ has an optimal control $v_T^*$. We extend this control by setting it to be zero over $(T,\infty)$. It is clear that this extended control, still denoted by $v_T^*$, satisfies that
\begin{equation}\label{m-1-48}
    \|v_T^*\|_{L^\infty(\mathbb{R}^+;L^2(\Omega))}=\alpha(T)\;\;\mbox{and}\;\;y(T;v_T^*,y_0)
    \in B(0,r).
\end{equation}
    It follows that $v_T^*\in\mathcal{U}_{\alpha(T)}$. By the optimality of $\tau(\alpha(T))$ for $(TP)_{\alpha(T)}$, together with the fact of $y(T;v_T^*,y_0)
    \in B(0,r)$, we have $\tau(\alpha(T))\leq T$. We now prove $\tau(\alpha(T))= T$. It is obvious that when $T=\gamma(y_0)$ then $\alpha(T)=0$ and $\tau(\alpha(\gamma(y_0)))=\gamma(y_0)$. Hence, without loss of generality, we assume that $T<\gamma(y_0)$.
\par
    Suppose that $\tau(\alpha(T))<T$. Since $\tau(\cdot)$ is continuous and strictly monotonically decreasing over $[0,+\infty)$, we can find $\widetilde{M}<\alpha(T)$ such that $\tau(\widetilde{M})=T$.
    From the optimality of $u^*_{\widetilde{M}}$ to Problem $(TP)_{\widetilde{M}}$, we get
\begin{equation}\label{m-1-49}
    \|\chi_{(0,T)}u^*_{\widetilde{M}}\|_{L^\infty(\mathbb{R}^+;L^2(\Omega))}
    =\|u^*_{\widetilde{M}}\|_{L^\infty(0,\tau(\widetilde{M});L^2(\Omega))}\leq \widetilde{M}<\alpha(T)
\end{equation}
    and
\begin{equation}\label{m-1-50}
    y(T;\chi_{(0,T)}u^*_{\widetilde{M}},y_0)=y(\tau(\widetilde{M});u^*_{\widetilde{M}},y_0)\in B(0,r).
\end{equation}
    These follow that
\begin{equation}\label{m-1-51}
    \chi_{(0,T)}u^*_{\widetilde{M}}\in\mathcal{V}_T.
\end{equation}
    This, together with the optimality of $\alpha(T)$ to Problem $(NP)_T$, yields that
\begin{equation}\label{m-1-52}
    \|u^*_{\widetilde{M}}\|_{L^\infty(0,T;L^2(\Omega))}\geq \alpha(T).
\end{equation}
    It is a contradiction to (\ref{m-1-49}). Then, the first identity in (\ref{m-theorem-1}) holds.
\par
    On the other hand, by the first identity in (\ref{m-theorem-1}), we know that $\tau(\alpha(\tau(M)))=\tau(M)$ for any $M>0$. This, together with the strict monotonicity of $\tau(\cdot)$, gives the second identity in (\ref{m-theorem-1}).
    The proof of Theorem \ref{m-theorem1} is now completed.
\end{proof}
\begin{remark}
    In Theorem \ref{m-theorem1}, we have shown that the function $\alpha(\cdot)$ is continuous as a function of $T$. For future research, other regularity results are also of interest, for example about the H\"older continuity of $\alpha(\cdot)$. For the hyperbolic case, results of this type have been presented in \cite{revision-2-1}.
\end{remark}
\par
    In the following, we shall present the proof of Theorem \ref{main-theorem}.
\begin{proof}[The proof of Theorem \ref{main-theorem}]
    Let $v^*_T$ be the optimal control to $(NP)_T$. Then, by the optimality of $v^*_T$ and the first identity in (\ref{m-theorem-1}), we can conclude that
\begin{equation}\label{m-1-53}
    y(\tau(\alpha(T));v^*_T,y_0)\in B(0,r)\;\;
    \mbox{and}\;\; \|v_T^*\|_{L^\infty(0,T;L^2(\Omega))}=\alpha(T).
\end{equation}
    We extend the control $v_T^*$ by zero to $(T,+\infty)$. Hence, the extended control, still denoted by $v_T^*$, is the optimal control to Problem $(TP)_{\alpha(T)}$.
\par
    On the other hand, if $u^*_M$ is the optimal control to Problem $(TP)_M$, then from the optimality of $u^*_M$ and the second identity in (\ref{m-theorem-1}), we get
\begin{equation}\label{m-1-54}
    y(\tau(M);u^*_M,y_0)\in B(0,r)\;\;\mbox{and}\;\;\|u^*_M\|_{L^\infty(\mathbb{R}^+;L^2(\Omega))}\leq
    \alpha(\tau(M)).
\end{equation}
    This implies that $\chi_{(0,\tau(M))}u^*_M$ is the optimal control to $(NP)_{\tau(M)}$.
    We complete the proof of this theorem.
\end{proof}
    Together with Theorem \ref{main-theorem} and Corollary \ref{p-corollary-1}, we have the following corollary:
\begin{corollary}\label{m-corollary1}
    Suppose that $(H_1)$ and $(H_2)$ hold. Then the optimal norm control $v_T^*$ to Problem $(NP)_T$ satisfies that
\begin{equation}\label{m-1-55}
    \|v_T^*\|_{L^2(\Omega)}=\alpha(T)\;\;\mbox{for almost all}\;\;t\in(0,T).
\end{equation}
\end{corollary}

\section*{Appendix: The proof of (\ref{p-1-2})}
\begin{proof}
     Indeed, when $u\equiv0$, multiplying the equation (\ref{m-1-1}) by  $y(t)\equiv y(t;0,y_0)$ and integrating on $\Omega$, we have
\begin{equation}\label{app-1-3}
    \frac{1}{2}\frac{d}{dt}\|y(t)\|^2_{L^2(\Omega)}+\|\nabla y\|^2_{L^2(\Omega)}+\langle f(y(t)),y(t)\rangle_{L^2(\Omega)}=0.
\end{equation}
    Let $\{\lambda_i\}_{i\in \mathbb{N}}$ and $\{e_i\}_{i\in \mathbb{N}}$ be the eigenvalue and eigenvector of $-\triangle$, i.e.,
$$
\begin{cases}
    -\triangle e_i=\lambda_ie_i&\mbox{in}\;\;\Omega,\\
    e_i=0&\mbox{on}\;\;\partial\Omega,
\end{cases}
    \;\;\mbox{for any}\;\;i\in\mathbb{N}.
$$
    From the results of Page 335, Chapter 6 in \cite{b14}, $0<\lambda_1\leq\lambda_2\leq \cdots$ and $\lambda_i\to+\infty$ as $i\to\infty$ hold. Moreover, $\{e_i\}_{i\in\mathbb{N}}$ forms an orthonormal basis of $L^2(\Omega)$. Therefore, for any $y\in H_0^1(\Omega)\subset L^2(\Omega)$, there exists a sequence $\{a_i\}_{i\in\mathbb{N}}$ such that
$$
    y=\sum_{i=1}^\infty a_ie_i
$$
    and
$$
    \lambda_1\|y\|^2_{L^2(\Omega)}\leq\sum_{i=1}^\infty\lambda_ia_i^2\leq\langle-\triangle y,y\rangle_{H^{-1}(\Omega),H_0^1(\Omega)}=\|\nabla y\|^2_{L^2(\Omega)},\;\;\forall y\in H_0^1(\Omega).
$$
    Hence, from the assumption $(H_1)$, the identity (\ref{app-1-3}) yields
\begin{equation}\label{app-1-4}
    \frac{1}{2}\frac{d}{dt}\|y(t)\|_{L^2(\Omega)}^2+\lambda_1\|y\|^2_{L^2(\Omega)}\leq 0.
\end{equation}
    This means that
\begin{equation}\label{app-1-5}
    \frac{d}{dt}\left(e^{2\lambda_1t}\|y(t)\|_{L^2(\Omega)}^2\right)\leq 0.
\end{equation}
    Therefore,
\begin{equation}\label{app-1-6}
    \|y(t)\|_{L^2(\Omega)}^2\leq e^{-2\lambda_1t}\|y_0\|^2_{L^2(\Omega)}.
\end{equation}
    This gives (\ref{m-1-2}).
\end{proof}

\vskip 10pt

\par
    \textbf{Acknowledgment.} The author thanks Professor Gengsheng Wang for stimulating discussions and comments. The author also gratefully acknowledges the anonymous referees for the suggestions which led to this improved version.

\end{document}